\documentclass[12pt,a4]{amsart}
\usepackage[a4paper, left=28mm, right=28mm, top=28mm, bottom=34mm]{geometry}
\usepackage{amsthm,amssymb}
\usepackage{amsmath}
\usepackage{cases}
\usepackage{here}
\usepackage[dvipdfmx]{graphicx}
\usepackage{amscd}
\usepackage{mathrsfs}
\newtheorem{thm}{Theorem}[section]

\newtheorem{lem}[thm]{Lemma}

\newtheorem{con}[thm]{Condition}
\def\P{\mathbb{P}}

\def\Z{\mathbb{Z}}
\def\F{\mathbb{F}}
\def\O{\mathcal{O}}
\def\T{\mathscr{T}}

\def\Spf{\mathop{\mathrm{Spf}}\nolimits}
\def\Sp{\mathop{\mathrm{Sp}}\nolimits}

\def\Spec{\mathop{\mathrm{Spec}}\nolimits}
\def\PGL{\mathop{\mathrm{PGL}}\nolimits}

\def\dist{\mathop{\mathrm{dist}}\nolimits}
\def\Mod{\mathop{\mathrm{mod}}\nolimits}
\def\vert{\mathop{\mathrm{vert}}\nolimits}
\def\edge{\mathop{\mathrm{edge}}\nolimits}
\def\Aut{\mathop{\mathrm{Aut}}\nolimits}
\def\Gal{\mathop{\mathrm{Gal}}\nolimits}
\def\sp{\mathop{\mathrm{sp}}\nolimits}
\def\val{\mathop{\mathrm{val}}\nolimits}
\def\id{\mathop{\mathrm{id}}\nolimits}

\numberwithin{equation}{section}
\numberwithin{figure}{section}

\makeatletter
\DeclareRobustCommand{\genericinterval}[2]{%
 \@ifstar{\genericinterval@star{#1}{#2}}{\genericinterval@nostar{#1}{#2}}}
\newcommand{\genericinterval@star}[4]{\mathopen{}\mathclose{\left#1#3,#4\right#2}}
\newcommand{\genericinterval@nostar}[4]{\mathopen{#1}#3,#4\mathclose{#2}}
\newcommand{\intervalcc}{\genericinterval[]}
\newcommand{\intervaloo}{\genericinterval][}

\newcommand{\intervalco}{\genericinterval[[}
\makeatother

\begin{document}
\title[Cyclic coverings of the projective line by Mumford curves]{Cyclic coverings of the projective line by Mumford curves  in positive characteristic}
\author{Ryota Mikami}
\address{Department of Mathematics, Faculty of Science, Kyoto University, Kyoto 606-8502, Japan}
\email{ryo-mkm@math.kyoto-u.ac.jp}
\date{\today}
\maketitle

\begin{abstract}
 We study the rigid analytic geometry of cyclic coverings of the projective line.
 We determine the defining equation of cyclic coverings of degree $p$ of the projective line by Mumford curves
over complete discrete valuation fields of positive characteristic $p$. Previously, Bradley studied that of  any degree over  non-archimedean local fields of characteristic zero.
\end{abstract}

\section{Introduction}
 A geometrically connected smooth projective curve of genus $\geq2$ over a complete discrete valuation field $(K, | \cdot | )$ is  called a \textit{Mumford curve} if it is analytically isomorphic to a rigid analytic space of the form  $(\mathbb{P}^1 \setminus \mathcal{L} )/ \Gamma$, where $\Gamma \subset \PGL_2(K)$ is a Schottky group and  $\mathcal{L} \subset \P^1$  is the set of limit points.
 Recall that  a finitely generated torsion-free discontinuous  subgroup of $\PGL_2(K)$  is called  a \textit{Schottky group} if it has infinitely many limit points in $\P^1$. 
   Mumford curves are algebraically characterized   by the property that they have split degenerate reduction  \cite[Theorem 3.3, Theorem 4.20]{Mum}. Cyclic coverings of $\P^1$ by Mumford curves were studied by Bradley and van Steen; see \cite{Bra07}, \cite{Ste82}.  When $K$ is a non-archimedean local field of characteristic zero, Bradley studied  the defining equation of cyclic coverings of any degree of $\P^1$ by  Mumford curves \cite[Theorem 4.3]{Bra07}.
 
 In this paper,  we focus on cyclic coverings of degree $p$ of $\P^1$ by Mumford curves  in  characteristic $p>0$. 
 Let $$\varphi \colon X \rightarrow \P^1$$ be a cyclic covering of degree $p$ over $K$.
 Assume that $K$ is of characteristic $p>0$. 
 In \cite[Proposition 3.1]{Ste82}, van Steen showed  that if $X$ is a Mumford curve,  by replacing $K$ by its finite extension, it is defined by an equation of the form
\begin{equation}\label{equation1}
y^p-y=\sum_{i=1}^{r}  \frac{\lambda_i}{x - a_i}
\end{equation}
for some $\lambda_i \in K^\times $ and $  a_i \in K 
\  (1\leq i \leq r )$ 
satisfying $ a_i \neq a_j \  \text{for }  \  i \neq j$.
In the following, we assume that $X$ is defined by the equation \eqref{equation1}.
The cyclic covering $X$ has genus $(p-1)(r-1)$ \cite[Proposition 1.3]{Ste82}.
Thus, we also assume $(p-1)(r-1) \geq 2$, i.e., $r\geq 3$, or $r =2$ and $p\geq 3$.
 The main theorem of this paper is the following:

\begin{thm} \label{main}
Let  $(K, | \cdot | )$ be  a complete discrete valuation field of positive characteristic $p>0$.
 Let $\varphi \colon X \rightarrow \P^1$ be a cyclic covering of degree $p$ over $K$ defined by the equation \eqref{equation1} for $r\geq 3$, or $r =2$ and $p\geq 3$. 
  Then the following conditions are equivalent:
 \begin{itemize}
 \item $X$ is a Mumford curve over a finite extension of $K$.
 \item $\lvert \lambda_i \lambda_j \rvert < \lvert a_i-a_j \rvert^2$ for any $i \neq j$.
  \end{itemize}  \end{thm}
 
  Previously,  van Steen studied the defining equation of  hyperelliptic curves which are Mumford curves \cite{Ste83}.
 When $p=r=2$,  the cyclic covering $X$ has genus $1$ and van Steen obtained results similar to Theorem \ref{main}; see  \cite[Section 4]{Ste83}.
 Tsushima told the author that for any $p$,   Theorem \ref{main} for $r=2$ can also be proved  by computing reductions explicitly.
 
  Note that for any cyclic covering $\varphi \colon X \to \P^1$ by a Mumford curve $X$ over $K$, 
  there exists a surjective homomorphism from a  discrete subgroup  of $\PGL_2(K)$ generated by finitely many elements of finite order to the Galois group of $\varphi$; see  \cite[Chapter 8]{GerritzenvanderPut:SchottkyMumford}.  
  Since the order of any element of finite order of $\PGL_2(K)$ is not divisible by $p^2$, 
   the degree $\deg \varphi$ is not divisible by $p^2$.
  When $p$ does not divide $\deg \varphi$,  we can use  Bradley's method  in \cite{Bra07} to study the defining equation of $X$.
     
  The organization of this paper is as follows. In Section 2, we review some basic properties of Mumford curves.
 In Section 3, we summarize some facts about cyclic coverings of $\P^1$ by Mumford curves proved by van Steen \cite{Ste82}. 
The proof of Theorem \ref{main} is given in Section 4 and Section 5.

\section{Basic properties of Mumford curves} 
In this paper, we use the language of rigid analytic geometry. We refer to \cite{Lutk} for basic notations on rigid analytic geometry and \cite{GerritzenvanderPut:SchottkyMumford} for those on Mumford curves used in this paper.

Let $(K, | \cdot |)$ be a complete discrete valuation field of characteristic $p>0$
and $K^{\circ}$ (resp.\ $k$) its valuation ring (resp.\ residue field). We fix a uniformizer $ \pi \in K^\circ $.
We fix an algebraic closure $\overline{K} $ of $K$. We also denote the extension of the  valuation $| \cdot |$ on $K$ to $\overline{K}$ by the same symbol.
 We denote by $\val_K(\cdot)$ the normalized additive valuation on $K$, i.e.,  we have $\val_K(\pi)=1$.
 
 Let $A$  be an affinoid algebra  over $K$. For an element $f \in A$, let $$\lvert f \rvert_{\sp}:=\sup \{\, \lvert f (x)\rvert \mid x \in \Sp A \, \}$$ be the \textit{spectral seminorm} of $f$. (It is called the \textit{supremum norm} in \cite[Section 1.4]{Lutk}.)
 We put  
 \begin{align*}
 A^\circ &:= \{\, f \in A \mid \lvert f \rvert_{\sp} \leq1\, \}, \\
 A^{\circ \circ} &:= \{\, f \in A \mid \lvert f \rvert_{\sp} <1\, \}.
 \end{align*}
We denote the residue ring of an affinoid algebra $A$ by $ \overline{A}:=A^\circ / A^{\circ \circ}$. The affine scheme $\Spec \overline{A}$ over $k$ is called  the \textit{canonical reduction} of the affinoid space $\Sp A$. We put $ \overline{\Sp A} := \Spec \overline{A}$. (For details, see \cite[Section 1.4]{Lutk}.)

For an affinoid algebra $A$  over $K$, an algebra $B$ of topologically finite type over $K^\circ$ is called a \textit{$K^\circ$-model} of $A$ if $B$ is flat over $K^\circ$ and $B\otimes_{K^\circ}K \cong A$; see \cite[Definition 3.3.1]{Lutk}.

 
 

 A subgroup $N $ of $ \PGL_2(K)$ is called \textit{discontinuous} if the set of limit points  of the canonical action of $N$ on $\P^1(K)$ does not equal to $\P^1(K)$ and the closure of $Na$ is compact for any $a \in \P^1(K)$. Obviously, a discontinuous subgroup is discrete. 
A finitely generated torsion-free discontinuous  subgroup of $\PGL_2(K)$  is called  a \textit{Schottky group} if it has infinitely many limit points in $\P^1$. 
  A Schottky group $\Gamma$ is a free group; see \cite[Chapter 1]{GerritzenvanderPut:SchottkyMumford}.
 We put $$\Omega := \P^1 \setminus \{ \, \text{the limit points of} \, \ \Gamma\}  ,$$
 which is a one-dimensional rigid analytic space over $K$.
   The quotient $\Omega / \Gamma$ is isomorphic to the analytification of a geometrically connected smooth projective curve $X_\Gamma$ of genus $\geq2$ over $K$. A smooth projective  curve of genus $\geq2$ over $K$ which is  isomorphic to $X_\Gamma$ for some Schottky group $\Gamma \subset \PGL_2(K)$  is called a \textit{Mumford curve}. We identify projective curves over $K$ and their analytifications by the ``GAGA''-correspondence.  Concerning the automorphism group, we have a natural isomorphism 
 $$\Aut ( X_\Gamma) \cong N_{\PGL_2(K)}(\Gamma) / \Gamma ,$$
  where $N_{\PGL_2(K)}(\Gamma)$ is the normalizer of $\Gamma$ in $\PGL_2(K)$; see \cite[Chapter 7]{GerritzenvanderPut:SchottkyMumford}.
  Mumford proved the following theorem: 

\begin{thm}[{Mumford \cite[Theorem 3.3, Theorem 4.20]{Mum}}] \label{Mumford}
A geometrically connected smooth projective  curve $X$ of genus $\geq2$ over $K$ is a Mumford curve if and only if it has split degenerate reduction, i.e., there exists a proper flat scheme $Y$ over $\Spec K^\circ $ such that
\begin{itemize}
\item $Y \times_{\Spec K^\circ} \Spec K \cong X$,
\item the normalizations of all the irreducible components of $Y \times_{\Spec K^\circ} \Spec \overline{k}$ are rational curves (where $\overline{k}$ is an algebraic closure of $k$), and 
\item all the singular points of the closed fiber $Y \times_{\Spec K^\circ} \Spec k$ are $k$-rational ordinary double points with   two $k$-rational branches.
\end{itemize} \end{thm} 

We collect some properties of the Bruhat-Tits tree $\T$ of $\PGL_2(K)$ used in Section \ref{only if} of this paper;  see \cite[Section 2]{CorKatoKont01}, \cite[Chapter II]{Tree} for details.
The Bruhat-Tits tree $\T$ is a combinatorial graph defined as follows:
\begin{itemize}
\item The set of vertices $\vert (\T )$ is the set of equivalence classes of $K^\circ$-lattices in $K \oplus K$. Here, two $K^\circ$-lattices $M_1, M_2$ are equivalent if 
$M_1= a M_2$ for some $ a \in K^\times$.
\item Two vertices $w_1, w_2 \in \vert (\T)$ are adjacent if and only if $\pi M_1 \varsubsetneq M_2 \varsubsetneq M_1$ for some  $K^\circ$-lattices $M_1$ and $ M_2$ in the equivalence classes $ w_1$ and $ w_2 $, respectively.
\end{itemize}

The graph $\T$ is actually a tree \cite[Chapter II, Theorem 1]{Tree}. The set of edges of  $\T$ is denoted by $\edge (\T)$.
A sequence  $w_1, w_2,w_3 \dots $ of distinct vertices of $\T$ gives a \textit{half-line} on $\T$ if $w_i $, $w_{i+1}$ are adjacent for any $i \geq1$.
Two half-lines given by  $w_1, w_2,w_3 \dots $ and $w_1', w_2',w_3' \dots $ are equivalent if there exist $i, j \geq 1$ such that  $w_{i+r}=w_{j+r}'$ for any $r\geq 0$.
An equivalence class of half-lines on $\T$ is called an \textit{end} of $\T$.
 There is a natural bijection between $\P^1(K)$ and  the set of ends of $\T$ as follows.
 For an element $a \in \P^1(K)$, let $V_a \subset K \oplus K$ be a $1$-dimensional $K$-subspace corresponding to $a$. 
 Let $w_i \in  \vert (\T)$ be the equivalence class of $K^\circ$-lattices containing $$ \pi^i (K^\circ \oplus K^\circ) + V_a \cap (K^\circ \oplus K^\circ).$$
Then the sequence $w_1, w_2,w_3 \dots $ gives the end of $\T$ corresponding to $a$. 
This bijection is equivariant  with respect to the action of $\PGL_2(K)$.
 See \cite[Chapter II, p.\,72]{Tree} for details.
 
For $ v,w \in \vert (\T )$ and $ a, b \in \P^1(K)$, let $ \intervalcc{v}{w}$ (resp.\ $ \intervalco{v}{a}  , \   \intervaloo{a}{b}  ) $ be the path from $v$ to $w$ without backtracking (resp.\ the half-line from $v$ to $a$, the line from $a$ to $b$), where we regard  $a,b$ as ends of $\T$.
For $ v,w \in \vert (\T )$, the length of the path $ \intervalcc{v}{w}$ is called the \textit{distance} from $v$ to $w$, 
and is denoted by $\dist (v,w)$; see \cite[Section 1.2]{Tree}.
For subtrees $\mathscr{R}, \mathscr{S} \subset \T$, we put 
$$\dist (\mathscr{R},\mathscr{S}) := \min_{ \substack {v \in \vert (\mathscr{R}) \\  w \in \vert (\mathscr{S})}} \dist(v,w).$$

We denote by  $v_1\in \vert (\T )$ the vertex  corresponding to the equivalence class of $K^\circ$-lattices containing $ K^{\circ} e_1 \oplus K^{\circ} e_2 $,  where $\lbrace e_1, e_2 \rbrace$ is the standard basis of $K \oplus K$.
   For  $a \in K^\times $ (resp.\ $w \in \vert (\T )$),   the intersection of  
   $  \intervaloo{0}{\infty} $, $   \intervaloo{0}{a} $, and $ \intervaloo{a}{\infty} $ (resp.\ $ \intervaloo{0}{\infty} $, $ \intervalco{w}{0} $, and $  \intervalco{w}{\infty} $) consists of one vertex only, and we denote it  by $v(0,\infty, a)$ (resp.\ $v(0,\infty ; w)$).
 
    \begin{itemize}
    \item If $\val_K (a) \geq 0 $, we have $v(0,\infty , a) \in  \intervalco{v_1}{0}$ and $ \dist(v_1, v(0,\infty , a) ) =  \val_K (a).$
    \item If $\val_K (a) \leq 0 $, we have $v(0,\infty , a) \in  \intervalco{v_1}{\infty}$ and $ \dist(v_1, v(0,\infty , a) ) = - \val_K (a).$
    \end{itemize}
   Since $ v(0,\infty , a)  = v(0,\infty ; w) $ for  any $ w \in \intervaloo{0}{a} \cap \intervaloo{a}{\infty} $, we can compute $ \val_K (a) $ by using $v(0,\infty ; w )$.
   
   For any discrete subgroup $N \subset \PGL_2(K)$ and any $v \in \vert (\T)$,  
   the stabilizer 
   $$N_v := \lbrace \, \gamma \in N \mid \gamma(v)=v \, \rbrace$$
     is a finite group.
   For an element  $\gamma \in \PGL_2(K) $ of finite order, let $M(\gamma ) \subset \T$ be the smallest subtree  generated by the vertices fixed by  $\gamma $. The subtree $M(\gamma)$ is called the \textit{mirror} of $\gamma$; see \cite[Section 2]{CorKatoKont01}. 
  An element $\gamma \in \PGL_2(K) $ of order $p$ is called a \textit{parabolic element}.
  A parabolic element $\gamma \in \PGL_2(K)$ has a unique fixed point in $\P^1$.
  For a parabolic element $\gamma \in \PGL_2(K)$ and  $v \in \vert (M(\gamma))$,  the subset
 $$ \lbrace\,  e \in  \edge (M(\gamma)) \mid v \text{ is an extremity of } e \, \rbrace $$
 consists of  one element only or coincides with $$\lbrace \, e \in \edge (\T) \mid v \text{ is an extremity of } e\,  \rbrace .$$
     For any parabolic element $\gamma \in \PGL_2(K)$ and any $v \in \vert (M(\gamma))$, the  element $\gamma$ acts freely on the following set:
    $$ \lbrace \, e \in \edge (\T) \setminus \edge (M(\gamma)) \mid v \text{ is an extremity of } e \, \rbrace .$$

     \begin{center}
   \includegraphics[width=7cm]{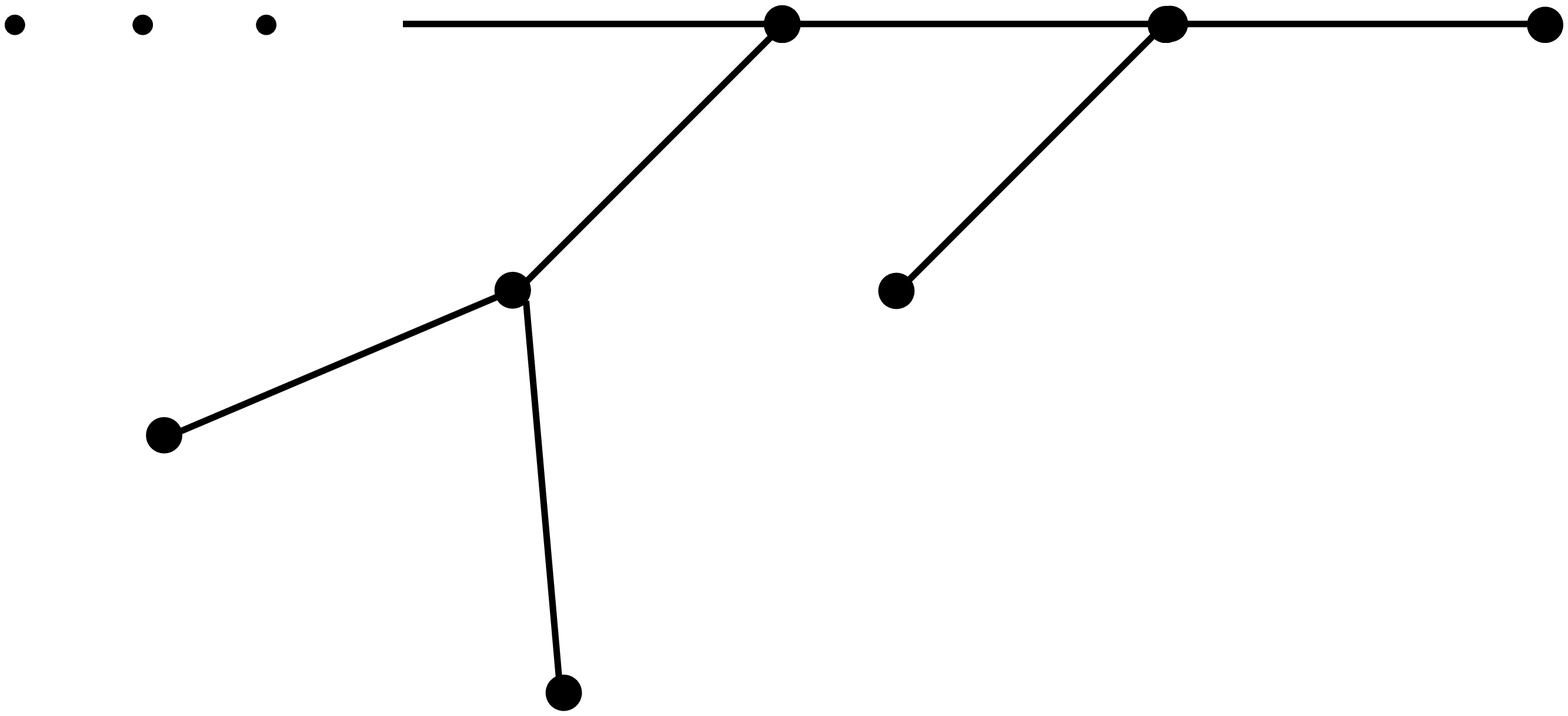}
   
   \textsc{Figure}.\  $M(\gamma)$ for a parabolic element $\gamma \in \PGL_2(K)$ when $k \cong \F_2$
     \end{center}
  

%
\section{Some facts about cyclic coverings of degree $p$ of the projective line by Mumford curves}\label{facts about cyclic coverings}
In this section, we review some facts about cyclic coverings of degree $p$ of $\P^1$ proved by van Steen \cite{Ste82}.
 Let $$\varphi \colon X \rightarrow \mathbb{P}^1$$ be a cyclic covering of degree $p$ over $K$. 
Let $ a_1,a_2, \dots , a_r \in  \P^1 $ be the branch points  of $\varphi$. 
We assume that $a_i \neq \infty $ for every $i$.
By replacing $K$ by its finite extension, we may assume that $a_1,a_2, \dots , a_r$ are $K$-rational points on $\P^1$.

We denote  the function field of $\P^1$ (resp.\ $X$) by $K(x)$ (resp.\ $F$). 
Since $F/K(x)$ is an Artin-Schreier extension,  
by replacing $K$ by its finite extension, 
there exists $ y \in F$ such that  $F = K(x,y)$ and 
$$y^p-y=\sum_{i=1}^{r} \sum_{ \substack{j=1 \\  j \not\equiv 0 \Mod p}}^{n_i} \frac{\lambda_{ij}}{(x - a_i)^j}$$
for some $\lambda_{ij} \in K^{\times}$.
Using this equation, we embed  $X$ into $\P^1 \times \P^1$.
If $X$ is a Mumford curve, we have $n_i=1$ for every $i$ \cite[Proposition 3.1]{Ste82}.  
We assume that $n_i =1$ for every $i$ and put $ \lambda_i := \lambda_{i1}$. 
The cyclic covering $X$ has genus $(p-1)(r-1)$ \cite[Proposition 1.3]{Ste82}.
Thus, we also assume $(p-1)(r-1) \geq 2$, i.e., $r\geq 3$, or $r =2$ and $p\geq 3$.
 Hence $X$ is defined by 
 $$y^p-y=\sum_{i=1}^{r}  \frac{\lambda_i}{x - a_i} .$$
 
If $X$ is a Mumford curve, 
 there exist $s_1, s_2 , \dots, s_r \in \PGL_2(K)$ satisfying the following conditions \cite[Proposition 2.2, Section 3]{Ste82}:
\begin{itemize}
\item  $s_i \ (1\leq i \leq r)$ is an element of order $p$,
\item the subgroup $N \subset \PGL_2(K)$ generated by $s_i \ (1\leq i \leq r)$  is discontinuous  and isomorphic to 
 the free product of $\langle s_i \rangle  \ ( 1\leq i \leq r)$, (this implies the subgroup $ \Gamma  \subset \PGL_2(K)$   generated by $s_i^n s_{i+1}^{-n}  $ $(1 \leq i \leq  r-1 , \ 1 \leq n \leq p-1)$  is a Schottky group satisfying   $N \subset N_{\PGL_2(K)}(\Gamma) $,)
  \item $X \cong \Omega / \Gamma,$  $\P^1 \cong \Omega / N,$ and the covering $\varphi \colon X \to \P^1 $ coincides with the natural projection $\Omega / \Gamma \to \Omega / N ,$
  where
  $$\Omega := \P^1\setminus \lbrace\, \text{the limit points of } \Gamma \, \rbrace ,$$
  \item the fixed point $P_i \in \P^1$ of $s_i$ is an element of $\Omega$,
 \item the image of  $P_i  $  under the natural projection $\Omega \to \Omega / N \cong \P^1$  is the branch point $a_i \in \P^1 $,
 \item $s_i(y)=y+1 \ (1\leq i \leq r)$, where we consider $s_i$ as an element of $\Aut ( X) \cong N_{\PGL_2(K)}(\Gamma) / \Gamma$.
 \end{itemize}

    In particular,  we have
  $$N/ \Gamma \cong \Gal(F/ K(x)) \cong \Z /p\Z .$$
We note that $M(s_i) \cap M(s_j) = \emptyset  $ for any $ i \neq j $.
In fact,  if there exists a vertex $v \in \vert (M(s_i) \cap M(s_j))$, it is fixed by infinitely many elements  $s_{l_1}^{n_1} \dots s_{l_m}^{n_m}$
 for $ m \geq 0, \ l_k \in \{i,j \} $, and $ 1\leq n_k \leq p-1 \ ( 1\leq k \leq m)$ with $l_k \neq l_{k+1} \ (1\leq k \leq m-1)$, but, since $N$ is discrete,  the stabilizer $N_v$ is a finite group.
 The contradiction shows $M(s_i) \cap M(s_j) = \emptyset  $.

\section{Proof of Theorem \ref{main} (part 1)}\label{if}
In this section,  we shall show that if the inequality
\begin{equation}\label{assumption}
\lvert \lambda_i \lambda_j \rvert < \lvert a_i-a_j \rvert^2
\end{equation}
 is satisfied for any $i \neq j $, 
 then $X$ is a Mumford curve over a finite extension of $K$.
 
By replacing $K$ by its finite extension,  for each $i$,  there exists  $\varepsilon_i  \in \lvert K^\times \rvert $ satisfying 
$$\varepsilon_i < \lvert a_i - a_j \rvert \ \text{and}  \ \lvert \lambda_i \rvert<  \frac{\lvert a_i - a_j \rvert^2}{\lvert \lambda_j \rvert}    - \varepsilon_i  $$
for any $j \neq i$.
By replacing $K$ by its finite extension, for $i,j,$ and $k$ satisfying $\lvert a_i -a_j \rvert < \lvert a_i -a_k \rvert$, 
there exists $\zeta_{i,j,k} \in \lvert K^\times \rvert$ satisfying $$\lvert a_i -a_j \rvert < \zeta_{i,j,k} < \lvert a_i -a_k \rvert.$$
 For each $1 \leq  i \leq r$, let 
 $  \alpha_{i,1} \leq \alpha_{i,2} \leq   \dots  \leq \alpha_{i,M_i -1}  $
 be  the following  elements  
 $$ \varepsilon_i,\   \lvert \lambda_i \rvert,\  \lvert a_i - a_j \rvert, \ 
 \frac{\lvert a_i - a_j \rvert^2}{\lvert \lambda_j \rvert}, \ 
    \frac{\lvert a_i - a_j \rvert^2}{\lvert \lambda_j \rvert}  - \varepsilon_i \ (j \neq i) , \
 \zeta_{i,j,k} \  (j,k  \ \text{satisfying} \     \lvert a_i -a_j \rvert < \lvert a_i -a_k \rvert )    $$
   of $\lvert K^\times \rvert $ arranged in ascending order.
    We put  $  \alpha_{i,0}=0$ and $ \alpha_{i,M_i} = \infty  $ for each $1 \leq i \leq r$.

            We define sets $I$ and $J$ by 
            \begin{align*}
      I & := \{ \, n= ( n_1,  \dots, n_r ) \in \Z^r     \mid   0 \leq n_i \leq M_i -1 \  (1 \leq i \leq r) \, \}, \\
      J &:= \{ \, n= ( n_1,  \dots, n_r ) \in I    \mid   \lvert a_i -a_j \rvert \neq \alpha_{i,n_i+1} \ \text{or}  \ \lvert a_i -a_j \rvert \neq\alpha_{j,n_j+1}\  \text{for any} \  i \neq j  \, \}.
\end{align*}


  For each $n = ( n_1,  \dots, n_r ) \in I$,  we define an affinoid open subvariety $U_n \subset \P^1$ by
  \begin{align}
  U_n:= 
  \lbrace \, x \in \P^1  \mid  \alpha_{i, n_i} \leq \lvert x-a_i \rvert \leq \alpha_{i,n_{i+1}}  \text{ for any }1 \leq i \leq r  \, \rbrace. \notag 
  \end{align}

  \begin{lem}\label{cover}
  $\{ U_n \}_{n \in J}$ is an affinoid covering of $\P^1$.
  \end{lem}
  \begin{proof}
   Since $\{ U_n \}_{n \in I}$ is an affinoid covering of $\P^1$,
    it  suffices to show that, for any $ n \in I \setminus J $ and any $ c \in U_n$,   
    there exists $ n' \in J$ satisfying $ c \in U_{n'}$.
  
  For $n \in I$, we put $$M_n := \sharp \{\, 1 \leq i \leq r  \mid \lvert a_i -a_j \rvert = \alpha_{i,n_i+1} \ \text{for some} \ j \neq i \, \}.$$ 
    We prove Lemma \ref{cover} by induction on $M_n.$
    
        We fix $ n \in I \setminus J $ and $ c \in U_n$.
    Since $n \in  I \setminus J $, there exist distinct elements $i,j $ satisfying  
$\lvert a_i -a_j \rvert = \alpha_{i,n_i+1}$ and $    \lvert a_i -a_j \rvert = \alpha_{j,n_j+1}$.
In particular, we have $M_n \geq 2$.
We have $$U_n \subset \{ \, x \in \P^1 \mid \lvert x -a_i \rvert \leq \lvert a_i -a_j \rvert \ \text{and} \ \lvert x -a_j \rvert \leq \lvert a_i -a_j \rvert \, \}.$$
Hence we have $\lvert c -a_i \rvert =\lvert a_i -a_j \rvert$ or $\lvert c -a_j \rvert = \lvert a_i -a_j \rvert$.

We may assume $\lvert c -a_i \rvert =\lvert a_i -a_j \rvert$.
 We put $n_k':= n_k$ for $k \neq i$ and $n_i':=n_i+l'$, where  we put $$l':= \min \{ \,  l \geq 1   \mid   \alpha_{i,n_i+1} <\alpha_{i,n_i+ l +1} \, \}.$$
  We put $n' := (n_1', \dots, n_r') \in I$.
 Then we have $c \in U_{n'}$.
 We have $$\alpha_{i,n_i+ l'+1} \leq \zeta_{i,j,k}  < \lvert a_i -a_k \rvert  $$ for  $k$ satisfying  $\lvert a_i -a_j \rvert < \lvert a_i -a_k \rvert $.
 Hence we have $\alpha_{i,n_i+ l'+1} \neq \lvert a_i -a_k \rvert $ for  $k \neq i$. 
 We have $M_{n'} <M_n$.
 By induction on $M_n$, there exists $n' \in J$ satisfying $c \in U_{n'}$.
  \end{proof}

   We put  $J':=  \lbrace  \, n \in J \mid U_n \neq \emptyset  \, \rbrace$. For each $n \in J'$, we put  $\beta_n := \min_{1\leq i \leq r } \alpha_{i, n_i +1} $. 
  
  For any $ n \in J'$, we have 
  $$ U_n = \{ \, x \in \P^1 \mid \lvert x - a_{l(n,0)} \rvert \leq \beta_n \, \} \setminus \bigcup_{\nu=0}^{N_n} D_{n,\nu},$$
  where
  \begin{align*}
 D_{n,\nu} & :=  \{ \, x \in \P^1 \mid \lvert x - a_{l(n,\nu)} \rvert <  \alpha_{l(n,\nu), n_{l(n,\nu)}} \, \} 
  \end{align*}
   for  some $N_n \geq 0$ and $1\leq l(n,\nu) \leq r \ (0\leq \nu \leq N_n)$ with $ \lvert a_{l(n,0)} - a_{l(n,\nu)} \rvert \leq \beta_n \ (1\leq \nu \leq N_n)$.
   We may assume  $D_{n,\nu} \cap D_{n,\nu'} = \emptyset $ for $\nu \neq \nu'$.
   We take $l(n,0)$ so that     $ \alpha_{l(n,0), n_{l(n,0)}} = \min_{0\leq \nu \leq N_n} \alpha_{l(n,\nu), n_{l(n,\nu)}}$.
 
   \begin{lem}\label{un}
  For each $n \in J'$, we have $\alpha_{l(n,\nu),n_{l(n,\nu)}} = \lvert a_{l(n,0)} - a_{l(n,\nu)} \rvert  $ for $1 \leq \nu \leq N_n$.
 Moreover, for $1 \leq \nu \leq N_n$, we have  
      $ \lvert a_{l(n,0)} - a_{l(n,\nu)} \rvert =  \alpha_{l(n,0), n_{l(n,0)}}$ or $ \lvert a_{l(n,0)} - a_{l(n,\nu)} \rvert = \beta_n $.
  \end{lem}
  \begin{proof}
   We fix $1 \leq \nu \leq N_n$. 
 Since $  a_{l(n,\nu)} \notin D_{n,0} $, we have $\alpha_{l(n,0), n_{l(n,0)}} \leq \lvert a_{l(n,0)} - a_{l(n,\nu)} \rvert$.
 We  have $$\lvert a_{l(n,0)} - a_{l(n,\nu)} \rvert \leq \beta_n \leq  \alpha_{l(n,0), n_{l(n,0)}+1} .$$
 
 Hence we have $\lvert a_{l(n,0)} - a_{l(n,\nu)} \rvert = \alpha_{l(n,0), n_{l(n,0)}}$ or $\lvert a_{l(n,0)} - a_{l(n,\nu)} \rvert = \alpha_{l(n,0), n_{l(n,0)}+1}$.
    Similarly, we have   $\lvert a_{l(n,0)} - a_{l(n,\nu)} \rvert = \alpha_{l(n,\nu), n_{l(n,\nu)}}$ or $\lvert a_{l(n,0)} - a_{l(n,\nu)} \rvert = \alpha_{l(n,\nu), n_{l(n,\nu)}+1}$.
  
  We assume that $ \lvert a_{l(n,0)} - a_{l(n,\nu)} \rvert \neq \alpha_{l(n,\nu), n_{l(n,\nu)}}$.
  Then we have $\lvert a_{l(n,0)} - a_{l(n,\nu)} \rvert = \alpha_{l(n,\nu), n_{l(n,\nu)}+1}$ and $\alpha_{l(n,\nu), n_{l(n,\nu)}}<\alpha_{l(n,\nu), n_{l(n,\nu)}+1}$.
  Since $\alpha_{l(n,0), n_{l(n,0)}} \leq \alpha_{l(n,\nu), n_{l(n,\nu)}}$,
  we  have $\lvert a_{l(n,0)} - a_{l(n,\nu)} \rvert = \alpha_{l(n,0), n_{l(n,0)}+1}$,
  which contradicts the definition of $J$. Hence we have $ \lvert a_{l(n,0)} - a_{l(n,\nu)} \rvert = \alpha_{l(n,\nu), n_{l(n,\nu)}}$.
  
    If  $\lvert a_{l(n,0)} - a_{l(n,\nu)} \rvert = \alpha_{l(n,0), n_{l(n,0)}+1}$, since 
  $ \lvert a_{l(n,0)} - a_{l(n,\nu)} \rvert \leq \beta_n $, we have $ \alpha_{l(n,0), n_{l(n,0)}+1} \leq \beta_n $. Hence we have $ \alpha_{l(n,0), n_{l(n,0)}+1} = \beta_n $.
  Consequently, we have  
      $ \lvert a_{l(n,0)} - a_{l(n,\nu)} \rvert =  \alpha_{l(n,0), n_{l(n,0)}}$ or $ \lvert a_{l(n,0)} - a_{l(n,\nu)} \rvert = \beta_n $.
  \end{proof}
  
    One can easily show that the admissible affinoid covering  $ \lbrace U_n \rbrace_{n \in J' }$ of $ \P^1$ is a \textit{formal analytic covering} in the sense of \cite[Definition 3.1.6]{Lutk}; see \cite[Proposition 2.2.6]{Fresnel-vanderPut:2004}.
  We see that  $\O(U_n) $ is reduced and  $\lvert \O (U_n) \rvert_{\sp} = \lvert K \rvert$; see \cite[Proposition 2.2.6]{Fresnel-vanderPut:2004}.
   Then  $\O(U_n)^\circ$  is a $K^\circ$-model of $\O(U_n)$ by \cite[Theorem 1 of Section 6.4.3]{BGR}.  
   Hence the formal analytic covering $ \lbrace U_n \rbrace_{n \in J' }$ of $ \P^1$ defines a proper admissible formal scheme  covered by $ \lbrace \Spf (\O(U_n)^\circ) \rbrace_{n \in J' }$ by  \cite[Theorem 3.3.12]{Lutk}. Hence it is  algebraic by Grothendieck's existence theorem.
    Consequently, the canonical reductions  $ \lbrace \overline{U_n} \rbrace_{n \in J' }$ define an algebraic reduction of $\P^1$.
    The canonical reductions  $ \lbrace \overline{\varphi^{-1} (U_n) } \rbrace_{n \in J'}$   define an algebraic reduction of $X$ over a finite extension of $K$.
      
 In order to show that  $X$ is a Mumford curve over a finite extension of $K$, it is enough to prove that, for each $ n \in J'$, the affinoid open subvariety $\varphi^{-1} (U_n) \subset X$ satisfies the following conditions over a finite extension of $K$:
 \begin{con}\label{conditionaffinoid} \
  \begin{itemize}
 \item All the irreducible components of the canonical reduction $\overline{\varphi^{-1}(U_n)}$ are rational curves, and 
 \item all the singular points of the canonical reduction $\overline{\varphi^{-1}(U_n)}$ are ordinary double points.
 \end{itemize}
 \end{con}
 
 We shall show that $\varphi^{-1} (U_n)$ satisfies Condition \ref{conditionaffinoid} by calculating the canonical reductions $\overline{\varphi^{-1} (U_n)}$ explicitly.
 We fix an element $n \in J'$.
 We put $ N:= N_n $ and $l := l(n,0)$. 
 We also put $ D_\nu := D_{n,\nu},  \ d_\nu := a_{l(n,\nu)}$ for each $0 \leq \nu \leq N$.
 We take $b_1 \in K$ and $b_2 \in K^\times \cup \{\infty\}$ satisfying $\lvert b_1 \rvert = \alpha_{l,n_l}$ and $ \lvert b_2 \rvert = \beta_n$.
  Then we have 
  \begin{align*}
  U_n &= \{ \, x \in \P^1 \mid \lvert x - d_0 \rvert \leq  \lvert b_2 \rvert \, \} \setminus \bigcup_{\nu=0}^{N_n} D_{\nu}, \\
 D_{0} & =  \{\, x \in \P^1 \mid \lvert x - d_0 \rvert <      \lvert b_1 \rvert \, \}  ,\\
 D_{\nu} & =  \{ \, x \in \P^1 \mid \lvert x - d_\nu \rvert <  \lvert d_0 - d_\nu \rvert   \, \}  \qquad (1 \leq \nu \leq N_n).
  \end{align*}
 For each $1 \leq \nu \leq N_n$, we  have  $\lvert d_0 - d_\nu \rvert = \lvert b_1 \rvert$ or $\lvert d_0 - d_\nu \rvert  = \lvert b_2 \rvert$.
 
 We put 
 $ z_i:= \lambda_i(x-a_i)^{-1} $ for each $1\leq i \leq r$.
  From the equation \eqref{equation1}, we have 
$$y^p-y = \sum_{i=1}^{r} z_i .$$
 We put
 \begin{align*}
 \Lambda := \{ \, i \mid  \lvert x-a_i \rvert \leq \lvert \lambda_i  \rvert \text{ for some } x \in U_n  \, \}
  = \{ \, i \mid \lvert z_i \rvert_{\sp} \geq 1 \text{ on } U_n  \, \}.
\end{align*}

\begin{lem}\label{lambdaempty}
If $\Lambda = \emptyset $, the affinoid open subvariety $\varphi^{-1}(U_n)$ satisfies  Condition \ref{conditionaffinoid}.
\end{lem}
\begin{proof}
Since $\Lambda = \emptyset $, we have $\lvert y \rvert_{\sp} \leq 1 $ on $\varphi^{-1} (U_n)$.
We embed $X$ into $\P^1 \times \P^1$ by $x $ and $y$.
 We have 
 $$\varphi^{-1} (U_n) \subset \{ \,(x,y) \in \P^1 \times \P^1 \mid x \in U_n , \ \lvert y \rvert \leq 1 \,\}.$$
Hence we have 
$$\O (\varphi^{-1} (U_n)) \cong \O (U_n)[y] / (y^p-y-\sum_{i=1}^r z_i).$$
Since $\O ( U_n)^\circ$ is a $K^\circ$-model of $\O(U_n)$,  the $K^\circ$-algebra 
$$\O  (U_n)^\circ [y]/ (y^p-y-\sum_{i=1}^r z_i)$$
 is a $K^\circ$-model of $\O (\varphi^{-1} (U_n))$.
Since $\lvert z_i \rvert <1 $ on $U_n$ for each $ 1 \leq i \leq r$, the residue ring $ \overline{ \O (\varphi^{-1} (U_n))}$ is isomorphic to $$\overline{ \O (U_n)}[y]/(y^p - y),$$ and   $\varphi^{-1}(U_n)$ satisfies  Condition \ref{conditionaffinoid}.
\end{proof}

 In the rest of this section, we assume $\Lambda \neq \emptyset.$ 

\begin{lem}\label{0one}
If $a_i \in U_n$ for some $1\leq i \leq r$, we have $$U_n = \{ \, x \in \P^1 \mid \lvert x - a_i \rvert \leq \alpha_{i,1} \,\}$$ and $a_j \notin U_n $ for $j \neq i$.
\end{lem}
\begin{proof}
Since $a_i \in U_n$, we have $n_i=0$.
For $j \neq i$, since $U_n $ is non-empty, we have 
$$ \{ \, x \in \P^1 \mid \lvert x - a_i \rvert \leq \alpha_{i,1}  \, \} \cap \{ \, x \in \P^1 \mid \alpha_{j,n_j} \leq \lvert x - a_j \rvert \leq \alpha_{j,n_j+1} \, \} \neq \emptyset .$$
Since $\varepsilon_i < \lvert a_i - a_j \rvert $, we have $ \alpha_{i,1} <  \lvert a_i - a_j \rvert$.
Hence we have 
   \begin{align*}
  &  \{ \, x \in \P^1 \mid \lvert x - a_i \rvert \leq \alpha_{i,1} \, \} \cap \{ \, x \in \P^1 \mid \alpha_{j,n_j} \leq \lvert x - a_j \rvert \leq \alpha_{j,n_j+1} \, \}  \\
  &= \{ \, x \in \P^1 \mid \lvert x - a_i \rvert \leq \alpha_{i,1}\, \} .
    \end{align*}
We have 
$$U_n = \bigcap_{j=1}^r \{ \, x \in \P^1 \mid \alpha_{j,n_j} \leq \lvert x - a_j \rvert \leq \alpha_{j,n_j+1} \, \} =  \{ \, x \in \P^1 \mid \lvert x - a_i \rvert \leq \alpha_{i,1}\, \}.$$
For $j \neq i$, since $a_j \notin \{ \,x \in \P^1 \mid \lvert x - a_i \rvert \leq \alpha_{i,1} \,\}$, we have $a_j \notin U_n$.
\end{proof}

By Lemma \ref{0one}, if $a_i \in U_n$ for some $i$, we have $b_1=0$ and $\lvert b_2 \rvert = \alpha_{i,1}$.

For each $a \in \P^1$, we put 
  $$\dist(a,U_n):= \inf_{u \in U_n} \lvert a-u \rvert .$$

\begin{lem}\label{aib1b2}
For $1 \leq i \leq r$, we have  $\dist(a_i,U_n) = \lvert b_1 \rvert $ or $\dist(a_i,U_n) \geq \lvert b_2 \rvert$.
\end{lem}
\begin{proof}
For $1\leq i \leq r$, by Lemma \ref{un}, we have 
$\dist(a_i,U_n)=0$,  $\dist(a_i,U_n) = \lvert b_1 \rvert $, or $\dist(a_i,U_n) \geq \lvert b_2 \rvert$.
Moreover, if  $\dist(a_i,U_n)=0$ for some $i$, we have $a_i \in U_n$, hence $b_1=0$.
\end{proof}

\begin{lem}\label{aib2}
If $\dist (a_i, U_n) > \lvert b_1 \rvert$ for some $i$,
we have $ \lvert x-a_i \rvert  = \lvert d_0-a_i \rvert $ 
for every $x \in U_n $. 
In particular, $  \lvert x-a_i \rvert = \dist(a_i,U_n) $ for every  $x \in U_n$.
\end{lem}
\begin{proof}
By Lemma \ref{aib1b2},  we have $\dist(a_i,U_n) \geq \lvert b_2 \rvert$.
First, we assume $\dist (a_i, U_n) = \lvert b_2 \rvert$. 
Then, by Lemma \ref{un}, we have $a_i \in D_{\nu(i)}$ for some $\nu(i)$ with $\lvert d_0 - d_{\nu(i)} \rvert =  \lvert b_2 \rvert$.
Since $\lvert a_i - d_{\nu(i)} \rvert < \lvert b_2 \rvert $, we have $ \lvert d_0 -a_i \rvert  = \lvert b_2 \rvert $.
We have
\begin{align*}
U_n 
& \subset   \{ \, x \in \P^1 \mid \lvert x - d_0 \rvert \leq \lvert b_2 \rvert  \ \text{and} \ \lvert x - d_{\nu(i)} \rvert \geq \lvert b_2 \rvert \,\} \\
& =  \{ \,x \in \P^1 \mid \lvert x - a_i \rvert = \lvert b_2 \rvert \,\}\\
& =  \{ \,x \in \P^1 \mid \lvert x - a_i \rvert = \lvert d_0 -a_i \rvert\,  \}.
\end{align*}

Next, we assume  $\dist (a_i, U_n) > \lvert b_2 \rvert$.
Then, for any $x \in U_n$, we have 
$$\lvert x-d_0 \rvert \leq \lvert b_2 \rvert < \dist(a_i,U_n) \leq  \lvert x-a_i \rvert .$$
Hence we have $\lvert x-a_i \rvert =\lvert d_0 - a_i \rvert$ for any $x \in U_n$.
\end{proof}

  \begin{lem}\label{mini}
  We have $\dist (a_i,U_n) \neq \dist(a_j,U_n)$ for any $i , j \in \Lambda $ with $i \neq j$.
    \end{lem}
  \begin{proof}
 Assume that we have $ \dist (a_i,U_n) =\dist(a_j,U_n)$ for some $i , j \in \Lambda $ with $i \neq j$. 
 We put  $d :=\dist (a_i,U_n) =\dist(a_j,U_n)  .$
 Then we have $d \leq \lvert \lambda_i \rvert $ and $d \leq \lvert \lambda_j \rvert$.
By Lemma \ref{aib1b2},  we have $ d= \lvert b_1 \rvert$ or $ d  \geq \lvert b_2 \rvert$.
By Lemma \ref{0one},    we have $d >0$.

First, we assume $d = \lvert b_1 \rvert$. 
By Lemma \ref{un}, we have $a_i \in D_{\nu(i)}$ for some $\nu(i)$ with $\lvert d_0 - d_{\nu(i)} \rvert = \lvert b_1 \rvert$.
Since $\lvert a_i - d_{\nu(i)} \rvert < \lvert b_1 \rvert$, we have $\lvert a_i - d_0  \rvert = \lvert b_1 \rvert$.
Similarly,  we have $\lvert a_j - d_0 \rvert = \lvert b_1 \rvert$.
Hence we have 
$$\lvert a_i - a_j \rvert \leq \max\{ \,\lvert d_0 - a_i \rvert , \ \lvert d_0 - a_j \rvert \,\} = \lvert b_1 \rvert = d \leq  \min \{ \,\lvert \lambda_i \rvert, \lvert \lambda_j \rvert \,\},$$
which contradicts the inequality \eqref{assumption}.

Next, we assume $d > \lvert b_1 \rvert$.
By Lemma \ref{aib2}, we have $\lvert x - a_i \rvert = d = \lvert x - a_j \rvert $ for any $x \in U_n$.
Hence 
we have 
$$\lvert a_i - a_j \rvert  \leq \max\{\, \lvert x - a_i \rvert , \ \lvert x - a_j \rvert \, \} =d  \leq \min \{\, \lvert \lambda_i \rvert, \ \lvert \lambda_j \rvert\, \}$$
for  any $ x\in U_n$,
  which contradicts the inequality \eqref{assumption}. 
  \end{proof}
   By Lemma \ref{mini}, there exists a unique element $m \in \Lambda $ satisfying 
 $$ \dist(a_m, U_n)  = \min_{i \in \Lambda} \dist(a_i,U_n).$$

\begin{lem}\label{z'i}
For any $i \in \Lambda \setminus \{ m \}$, we have 
  $$\left| \frac{\lambda_i}{x-a_i} +\frac{\lambda_i}{a_i-a_m} \right|   <1  $$
  for every $x \in U_n$.
\end{lem}
\begin{proof}
 Since $ \dist(a_m, U_n) <  \dist(a_i, U_n)$, by Lemma \ref{aib1b2}, 
 we have $ \dist(a_i, U_n) > \lvert b_1 \rvert $.
 By Lemma \ref{aib2}, we have 
 $\lvert x - a_i \rvert = \dist(a_i, U_n) $ for every $x \in U_n$.
 For  $x \in U_n$  satisfying $\lvert x - a_m \rvert = \dist(a_m, U_n)$,
we have   $\lvert x - a_m \rvert   < \lvert x - a_i \rvert .$
 Hence we have $\lvert x - a_i \rvert=\lvert a_m- a_i \rvert$  for every $x \in U_n$.
 
 Since $m \in \Lambda$, we have $\alpha_{m,n_m}\leq \lvert \lambda_m \rvert $.
 Since $\lvert \lambda_m \rvert <\lvert a_m - a_i \rvert^2 \cdot \lvert \lambda_i \rvert^{-1}  - \varepsilon_i $,
 we have  $\alpha_{m,n_m+1} \leq \lvert a_m - a_i \rvert^2 \cdot \lvert \lambda_i \rvert^{-1}  - \varepsilon_i $.
 Hence we  have
$\lvert x- a_m \rvert  \leq  \lvert a_m - a_i \rvert^2 \cdot \lvert \lambda_i \rvert^{-1}  - \varepsilon_i $
for every $x \in U_n$.

Consequently,  for every  $x \in U_n$, we have 
\begin{align*}
    \left| \frac{\lambda_i}{x-a_i} +\frac{\lambda_i}{a_i-a_m} \right| 
&  =  \frac{\lvert \lambda_i \rvert \cdot \lvert x-a_m \rvert}{\lvert x-a_i \rvert  \cdot \lvert a_i-a_m \rvert }  \\
&=  \frac{\lvert \lambda_i \rvert }{\lvert a_m-a_i \rvert^2 } \bigg(   \frac{ \lvert a_m-a_i \rvert^2}{\lvert \lambda_i \rvert } - \varepsilon_i \bigg) \\
& <1.
\end{align*}
\end{proof}

We put 
\begin{align*}
f &:= \sum_{i \in \Lambda \setminus \{m\}} \bigg(z_i + \frac{\lambda_i}{a_i-a_m} \bigg) + \sum_{i \notin \Lambda} z_i ,\\
C&:= -\sum_{i \in \Lambda \setminus \{m\}} \frac{\lambda_i}{a_i-a_m}.
\end{align*}
Then we have 
 $$y^p-y = \sum_{i=1}^{r} z_i = z_m  + C + f .$$

By Lemma \ref{z'i}, we have  $\lvert f \rvert_{\sp} < 1 $ on $U_n$.

\begin{lem}\label{z}
There exist  $ b_1'  \in  K^\times , \  b_2' \in  K^\times  \cup \{ \infty \}$, and  $C'  \in K$ satisfying 
$ \lvert b_1' \rvert \leq \lvert b_2' \rvert  \leq1 $ or $ 1 \leq \lvert b_1' \rvert \leq \lvert b_2' \rvert  $, and  
 $$U_n= \{ \, x \in \P^1 \mid  \lvert b_1'\rvert \leq \lvert z_m+C' \rvert \leq  \lvert b_2' \rvert \, \}  \setminus  \bigcup_{\nu =1}^N D_ {\nu}',$$ 
  where
  $$D_{\nu}' = \{\, x \in \P^1 \mid \lvert z_m + C'-d_{\nu}' \rvert < \lvert d_{\nu}' \rvert \,\}, $$ 
  for some   $  d_\nu' \in K^\times $ 
  with    $ \lvert d_\nu' \rvert = \lvert b_1' \rvert $ or $ \lvert d_\nu' \rvert = \lvert b_2' \rvert $.
\end{lem}
\begin{proof}
If  $\dist(a_m,U_n)= \lvert b_1 \rvert $, we put $  b_1':= \lambda_m b_2^{-1}, \ b_2':= \lambda_m b_1^{-1}$, and $ C' :=0$. (We put $b_2':= \infty$ if $b_1=0$.)
If  $\dist(a_m,U_n) \neq  \lvert b_1 \rvert $,  we put $b_1' := \lambda_m b_1(a_m-d_0)^{-2}$, $b_2' := \lambda_m b_2(a_m-d_0)^{-2}$, and
   $ C' :=\lambda_m(a_m-d_0)^{-1}$.
  In  both cases, we can check  $b_1', b_2'$, and $C'$  satisfy  the conditions of Lemma \ref{z}. Since the computations are straightforward, we omit them.
   \end{proof}

We put $z:= z_m +C'$. We regard it as a coordinate function on $\P^1$. 
By replacing $K$ by its finite extension, there exists $C'' \in K$ satisfying $C''^p-C''= C-C'.$
We put $ y':= y-C''$.
Then we have $y'^p-y' =z+ f $.

If $b_2' \neq \infty$, since 
  $$\O( \lbrace \, z \in \P^1 \mid \lvert b_1' \rvert \leq \lvert z \rvert \leq   \lvert b_2' \rvert \, \rbrace) 
  \cong K \langle b_2'^{-1}z , b_1' z^{-1}  \rangle,$$ 
 the residue ring   $ \overline{ \O (U_n)} $  is isomorphic to a localization of $k[s, t]/(st- \overline{b_1'b_2'^{-1}})$.
 
 If $b_2' = \infty$, since 
  $$\O( \lbrace \, z \in \P^1 \mid \lvert b_1' \rvert \leq \lvert z \rvert   \, \rbrace) 
  \cong K \langle b_1' z^{-1}  \rangle,$$ 
 the residue ring   $ \overline{ \O (U_n)} $  is isomorphic to a localization of $k[t]$.

We consider the following two cases separately:
\begin{itemize}
\item  $ \lvert b_1' \rvert \leq \lvert b_2' \rvert  \leq1 $.
\item $ 1 \leq \lvert b_1' \rvert \leq \lvert b_2' \rvert  $.
\end{itemize}

 \begin{lem}
If $\lvert b_1' \rvert \leq \lvert b_2' \rvert  \leq1 $,
the affinoid open subvariety $\varphi^{-1}(U_n)$ satisfies  Condition \ref{conditionaffinoid} over a finite extension of $K$.
\end{lem}
\begin{proof}
Since $  \lvert b_2' \rvert  \leq1$, we have $\lvert z \rvert_{\sp} \leq 1$ on $U_n$. 
Similarly to the proof of Lemma \ref{lambdaempty},
 the residue ring $ \overline{ \O (\varphi^{-1} (U_n))} $ is isomorphic to a localization of 
$$ k[s, t, y']/(st - \overline{b_1'b_2'^{-1}}, y'^p - y' - \overline{b_2'}s)$$

If $\lvert b_2' \rvert <1 $, we have $\overline{b_2'} =0$,
   and  $\varphi^{-1}(U_n)$ satisfies  Condition \ref{conditionaffinoid}.
    
    If $\lvert b_2' \rvert  =1 $, we have
   $$ k[s, t, y']/(st - \overline{b_1'b_2'^{-1}}, y'^p - y' - \overline{b_2'}s)  \cong k[t, y']/( t(y'^p-y') - \overline{b_1'}),$$ 
  and  $\varphi^{-1}(U_n)$ satisfies  Condition \ref{conditionaffinoid}.
    \end{proof}

\begin{lem}
If $ 1 \leq \lvert b_1' \rvert \leq \lvert b_2' \rvert $,
the affinoid open subvariety $\varphi^{-1}(U_n)$ satisfies Condition \ref{conditionaffinoid} over a finite extension of $K$.
\end{lem}
\begin{proof} 
Since $ 1 \leq \lvert b_1' \rvert$, we have $ 1 \leq \lvert z \rvert$ on $U_n$.
By replacing $K$ by its finite extension, there exists $\xi,  \xi' \in K$ such that $ \xi ^p =  b_2'^{-1}$ and $ \xi'^p= b_1'^{-1}$. (Here, we put $b_2'^{-1}: = 0$ if $b_2' = \infty$.)
 We put $f':= b_2'^{-1} f$ and  $f'' := z^{-1}f$. 
 Since $1  \leq \lvert b_2' \rvert$, we have $\lvert f' \rvert_{\sp}<1$ on $U_n$.  Since $ 1 \leq \lvert z \rvert$ on $U_n$, we have $ \lvert f'' \rvert_{\sp} < 1$ on $U_n$.
 We also put  $y'':=\xi y' $ and $w:=  \xi'^{-1} y'^{-1} .$
    Then we have
     \begin{align*}
      y''^p -  \xi^{p-1} y''   & = b_2'^{-1} z+ f' \\
         b_1' z^{-1}
   (1- \xi'^{p-1} w^{p-1}) &  = w^p (1 + f'').
    \end{align*}

    First,  we assume $b_2' \neq \infty$.   
Similarly to the proof of Lemma \ref{lambdaempty}, 
the residue ring  $\overline{ \O (\varphi^{-1}(U_n))} $  is isomorphic to a localization of 
 $$k[s,t,y'',w]/(st- \overline{b_1' b_2'^{-1}}, y''w-\overline{\xi \xi'^{-1}},
  y''^p-\overline{\xi} ^{p-1}y''-s, t(1-\overline{\xi'}^{p-1} w^{p-1}) -w^p  ),$$ 
  which is a localization of $k[y'',w]/(y''w-\overline{\xi \xi'^{-1}}),$
  and  $\varphi^{-1}(U_n)$ satisfies  Condition \ref{conditionaffinoid}.

Next,    we assume $b_2' =\infty$.
    Similarly to the proof of Lemma \ref{lambdaempty}, the residue ring  $\overline{ \O (\varphi^{-1}(U_n))} $  is isomorphic to a localization of 
 $$k[t,w]/(t(1-\overline{\xi'}^{p-1} w^{p-1}) -w^p  ),$$ 
  which is a localization of $k[w]$,
  and  $\varphi^{-1}(U_n)$ satisfies  Condition \ref{conditionaffinoid}.
     \end{proof}

Consequently, $X$ is a Mumford curve over a finite extension of $K$.

\section{Proof of Theorem \ref{main} (part 2)}\label{only if}
In this section, we shall show that if $X$ is a Mumford curve, the inequality $\lvert \lambda_i \lambda_j \rvert < \lvert a_i-a_j \rvert^2$ is satisfied for any $i \neq j$.  
 Since the assertion is symmetric, we need only to prove  the inequality
$$\lvert \lambda_1 \lambda_2 \rvert < \lvert a_1-a_2 \rvert^2.$$
We use van Steen's method in \cite[Section 3]{Ste83} and the Bruhat-Tits tree $\T $ of $\PGL_2(K)$.

 Take $s_1,s_2,\dots, s_r \in \PGL_2(K)$ as in Section \ref{facts about cyclic coverings} of this paper. 
    By replacing $K$ by its finite extension, we may assume that all the fixed points of $N$ on $\Omega$ are $K$-rational points.
   
 Let $M  \subset  \T$ be 
    the subtree generated by $M(s_i) \ (1 \leq i \leq r)$.  
 For each $i \neq j $, since $M(s_i) \cap M(s_j) = \emptyset$, there exist unique vertices 
 $\xi_i(j) \in \vert (M(s_i))$ and $  \xi_j(i) \in \vert( M(s_j))$ satisfying 
     $$\dist (M(s_i), M(s_j))= \dist (\xi_i(j), \xi_j(i)).$$
   For each $i \neq j$, let  $e_i(j) \in  \edge (  \intervalcc{\xi_i(j)}{\xi_j(i)} )$ be  the edge such that  $\xi_i(j)$ is an extremity of $e_i(j)$.   
 

 \begin{lem}\label{graph}
 There exist $  s_i'  \in N \ (1 \leq i \leq r) $  satisfying the following conditions:
 \begin{itemize}
 \item For each $i$, the element $s_i' $ is $N$-conjugate to $s_i$.
 (This implies $s_i'$ is an element of order $p$ with $s_i'(y) = y +1 $.)
 \item $N$ is the free  product of $\langle s_i' \rangle \ (1 \leq i \leq r)$. 
 (This implies $\Gamma $ is generated by $ s_i'^n s_{i+1}'^{-n} \ (1\leq i \leq r-1, \  1\leq n \leq p-1)$.)
 \item We have $e \neq s_i'^n(e')$ for any  $ 1 \leq i \leq r$,  $ 0 \leq n \leq p-1$, and distinct edges $e , e' \in \edge (M')$, where $M' \subset \T$  is the subtree generated by $M(s_i') \ (1 \leq i \leq r)$.
   \end{itemize} 
 \end{lem}
 \begin{proof}
 We prove Lemma \ref{graph} by induction on 
 $$\sum_{1 \leq i , j \leq r} \dist (M(s_i), M(s_j)).$$
 
 Since $M(s_i) \cap M(s_j) = \emptyset$ for $i \neq j$, we have $$\sum_{1 \leq i , j \leq r} \dist (M(s_i), M(s_j)) \geq r(r-1).$$

    We assume $e = s_m^n(e')$ for some distinct elements $e , e' \in \edge (M)$, $ 1 \leq m \leq r$, and $  1\leq n \leq p-1$.
  We fix $v_m \in \vert (M(s_m))$.
    There exists an extremity $v'$ of $e'$ with $v' \notin \vert( M(s_m))$. The vertex $s_m^n(v')$ is an extremity of $e$.
    There exist   $k,l \in \{ 1, \dots , r \} \setminus \{m \}$ satisfying
    $e_m(k) \in \edge (\intervalcc{v_m}{s_m^n(v')})$ and  $e_m(l) \in \edge (\intervalcc{v_m}{v'})$.
    We have $e_m(k) = s_m^n(e_m(l))$.
     In particular, we have $\xi_m(k) = \xi_m(l) $.
  
   For $i,j \in \{ 1, \dots , r \} \setminus \{m \}$ with $e_m(i)  \neq e_m(j)$,
   we have 
   $$\edge (\intervalcc{\xi_i(m)}{\xi_m(i)} \cap \intervalcc{\xi_m(j)}{\xi_j(m)}) = \emptyset,$$
   hence we have
    $$\intervalcc{\xi_i(m)}{\xi_j(m)} = \intervalcc{\xi_i(m)}{\xi_m(i)} \cup \intervalcc{\xi_m(i)}{\xi_m(j)} \cup \intervalcc{\xi_m(j)}{\xi_j(m)}.$$
  In particular,   we have $\intervalcc{\xi_i(m)}{\xi_j(m)}  \cap M(s_m) \neq \emptyset.$
     Hence, for $i,j \in \{ 1, \dots , r \} \setminus \{m \}$ with $e_m(i)  \neq e_m(j)$, we have $\xi_i(j)  = \xi_i(m) , $ $    \xi_j (i) = \xi_j(m) ,$ and
   \begin{align*}
\dist (M(s_i),M(s_j)) 
&  =\dist (\xi_i(j), \xi_j(i)) \\
& = \dist (\xi_i(m), \xi_m(i)) + \dist(\xi_m(i),  \xi_j(m)).
\end{align*}

 For $i \neq m$, we put 
 $$ I_i:= \{ \, 1 \leq j \leq r  \mid j \neq m  \text{ and } e_m(j) =e_m(i)\, \} . $$
 Then, for each $i \in I_k$ and $j \in I_l$, we have  
 $$e_m(i)= e_m(k) = s_m^n (e_m(l)) =s_m^n (e_m(j))  \in \edge( \intervalcc{\xi_m(j)}{s_m^n(\xi_j(m))} )$$
  and $\xi_m(i) = \xi_m(j) $.
 Hence we have $$e_m(i) \in  \edge( \intervalcc{\xi_i(m)}{\xi_m(i)} ) \cap \edge( \intervalcc{\xi_m(i)}{s_m^n(\xi_j(m))} ).$$
  
For each $i \in I_l$,  we put $s_i' := s_m^n s_i s_m^{-n}$. For each $i \not\in I_l$,  we put $s_i' :=  s_i $.
Then the discrete subgroup $N \subset \PGL_2(K)$ is the  free product of  $\langle s_i' \rangle \ (1 \leq i \leq r)$.
We have $M(s_i') = s_m^n M(s_i)$ for $i \in I_l$. 

We shall  show 
$$\sum_{1\leq i ,j \leq r}  \dist (M(s_i'), M(s_j')) <  \sum_{1 \leq i ,j \leq r} \dist (M(s_i), M(s_j)) .$$

To prove the above inequality, we estimate $ \dist (M(s_i'), M(s_j'))$ for each $i,j$.

\begin{itemize}
\item For $i \in I_k$ and  $j \in I_l$, we have $e_m(i) \neq e_m(j)$. 
  We have
\begin{align*}
\dist (M(s_i'),M(s_j')) & = \dist (M(s_i),s_m^nM(s_j)) \\
&   \leq \dist (\xi_i(m), s_m^n \xi_j(m)) \\
& <  \dist (\xi_i(m), \xi_m(i)) + \dist(\xi_m(i), s_m^n \xi_j(m)) \\
  & =  \dist (\xi_i(m), \xi_m(i)) + \dist(\xi_m(i),  \xi_j(m))  \\
  &= \dist (M(s_i),M(s_j)).
\end{align*}
\item For  $i ,j \in I_l$, we have 
\begin{align*}
\dist (M(s_i'),M(s_j'))=\dist (s_m^nM(s_i),s_m^nM(s_j)) = \dist (M(s_i),M(s_j)).
\end{align*}
\item For  $i=m$ and  $j \in I_l$, we have
\begin{align*}
\dist (M(s_m'),M(s_j')) =\dist (M(s_m),s_m^nM(s_j)) = \dist (M(s_m),M(s_j)).
\end{align*}
\item For  $i \not\in I_k \cup I_l \cup \{ m \}$ and  $j \in I_l$, since $e_m(i) \neq e_m(j)$ and $ e_m(i) \neq e_m(k)= s_m^n e_m(j) $,
 we have
\begin{align*}
\dist (M(s_i'),M(s_j')) & =\dist (M(s_i),s_m^nM(s_j)) \\
& = \dist (\xi_i(m), \xi_m(i)) + \dist(\xi_m(i), s_m^n \xi_j(m))\\
& = \dist (\xi_i(m), \xi_m(i)) + \dist(\xi_m(i),  \xi_j(m))\\
& =\dist (M(s_i),M(s_j)) .
\end{align*}
\item For  $i ,j \not\in I_l$, since $s_i'=s_i$ and $s_j' =s_j$,
we have
\begin{align*}
\dist (M(s_i'),M(s_j'))& =\dist (M(s_i),M(s_j)).
\end{align*}
\end{itemize}
Consequently, we have  $$\sum_{1\leq i ,j \leq r} \dist (M(s_i'), M(s_j')) <  \sum_{1 \leq i ,j \leq r} \dist (M(s_i), M(s_j)) .$$
By induction,  there exist  $s_i' \in N \ (1 \leq i \leq r)$ satisfying the conditions of Lemma \ref{graph}.
 \end{proof}
 
 We replace $s_i$ by $s_i'$ for every $1 \leq i \leq r$.
  Then we have $e \neq s_i^n(e')$ for any  $ 1 \leq i \leq r$, $  0 \leq n \leq p-1$, and any distinct elements $e,  e' \in M$. 
 
 Recall that we put $v_1:= v(0,\infty, 1)$, and $P_i \in \Omega $ is the fixed point of $s_i$ for $1\leq i \leq r$.
  By  replacing $K$ by its finite extension and changing the coordinate of $\Omega \subset \P^1$,
  we may assume that the following conditions are satisfied: 
 \begin{itemize}
\item  $P_1=0$ and $ P_2 \neq \infty $.
\item$ \lvert P_i \rvert < \lvert P_2\rvert  \text{ for any }i \ne 2$.
\item The element $s_1 \in \PGL_2 (K)$ is written as
 \begin{align*}
      s_1 & =  \begin{pmatrix} 1 & 0 \\ 1 & 1 \end{pmatrix} .
    \end{align*}
\end{itemize}
Then we have $ M( s_1)  \cap \intervaloo{0}{\infty}  =  \intervalco{v_1}{0}$.

Since $s_2 \in \PGL_2(K)$ is an element  of order $p$ fixing  $P_2 \in \P^1(K)\setminus \{0, \infty\} =K^\times$, it is  written as 
  $$ s_2 =  \begin{pmatrix} P_2(P_2-\eta) & \eta P_2^2 \\ -\eta & P_2(P_2+ \eta) \end{pmatrix} $$
for some $\eta \in K^\times$.
 
 \begin{lem}\label{v}
We have 
$$\val_K(\eta)  = - \dist (M(s_1), M(s_2))<0.$$
 In particular, we have $\lvert \eta \rvert >1$.
\end{lem}
\begin{proof}
Let  
\begin{align*}
\gamma := \begin{pmatrix} P_2 & 0 \\ -1 & P_2 \end{pmatrix} .
\end{align*}
Then we have
 \begin{align*}
 &\gamma s_1 \gamma^{-1}=   \begin{pmatrix} 1 & 0 \\ 1 & 1 \end{pmatrix}   , \\
& \gamma s_2 \gamma^{-1}=  \begin{pmatrix} 1 & \eta \\ 0 & 1 \end{pmatrix}  
 \end{align*}
in $\PGL_2(K) $.
We have
$ M(\gamma s_1\gamma^{-1})  \cap \intervaloo{0}{\infty}  =  \intervalco{v_1}{0}$ and 
$ M(\gamma s_2\gamma^{-1})  \cap \intervaloo{0}{\infty}  =  \intervalco{v(0,\infty,\eta)}{\infty}$.
Since $M(s_1) \cap M(s_2) = \emptyset$ and $M(\gamma s_i \gamma^{-1})  = \gamma M(s_i) \ (i=1,2) $,
we have 
$ M(\gamma s_1 \gamma^{-1}) \cap M(\gamma s_2 \gamma^{-1}) = \emptyset.$
Hence we have
 $\lvert \eta \rvert >1$ and
 \begin{align*}
  \val_K (\eta)  =- \dist (M(\gamma s_1 \gamma^{-1}), M(\gamma s_2 \gamma^{-1}))
  = - \dist (M(s_1), M(s_2))<0.
  \end{align*}  
 \end{proof}
 
Since $\val_K(\eta)  = - \dist (M(s_1), M(s_2))$ is invariant under  $\PGL_2(K)$-conjugation, we may also assume $\lvert \eta \rvert < \lvert P_2 \rvert $.
 Since 
$ M(s_1)  \cap \intervaloo{0}{\infty}  =  \intervalco{v_1}{0},$
 we have $\xi_1(2) =v_1 $, $\xi_2(1) = v(0,\infty, \eta)$, and 
$$ \intervalcc{v(0,\infty, \pi P_2)}{v(0,\infty, P_2 )}  \subset M(s_2).$$

 \begin{lem}\label{xiij}
 For any $i \neq j$, we have 
$$v(0,\infty; \xi_i(j)) \in \vert(  \intervalco{v(0,\infty, \pi P_2)}{0} ).$$
 \end{lem}
 \begin{proof}
 For  $ i=2$ and $ j \neq 2$, since $\lvert P_j \rvert <  \lvert P_2 \rvert $ and
$$ \intervalcc{v(0,\infty, \pi P_2)}{v(0,\infty,  P_2 )}  \subset M(s_2),$$
we have 
$$v(0,\infty; \xi_2(j)) \in \vert(  \intervalco{v(0,\infty, \pi P_2)}{0} ).$$

For $i \neq 2$, 
since $\lvert P_i \rvert <  \lvert P_2 \rvert $ and $ v(0,\infty ,  P_2 )  \in \vert ( M(s_2)),$
 we have
$$v(0,\infty; w ) \in \vert(  \intervalco{v(0,\infty, \pi P_2)}{0} )$$
for  $w \in \vert (M(s_i))$.
In particular, for  $i \neq 2$ and $j \neq i$, we have
$$v(0,\infty; \xi_i(j)) \in \vert(  \intervalco{v(0,\infty, \pi P_2)}{0} ).$$
\end{proof}

  By replacing $K$ by its finite extension,
 there exists a $K$-rational point $u \in \Omega $ such that $ \lvert u \rvert  =  \lvert u- P_2\rvert = \lvert  P_2\rvert.$

  \begin{lem}\label{huti}
    The following are satisfied:
   \begin{enumerate}
     \item For $1 \leq n \leq p-1$,  we have $ \lvert s_2^n(P_1) \rvert =\lvert \eta \rvert $. 
     \item  For $1 \leq n \leq p-1$,  we have $ \lvert s_2^n(u) \rvert =\lvert P_2 \rvert $.
   \item  For $1\leq n \leq p-1$, we have $ \lvert s_1^n(u) \rvert =1 .$
    \item For  any $\gamma \in N$ and $i \neq 2$, we have  $\lvert \gamma(P_i) \rvert   < \lvert P_2 \rvert $. 
      \item For any  $\gamma \in N$,  we have $\lvert \gamma(u)- P_2 \rvert  = \lvert P_2 \rvert $.
     \item For any  $\gamma \in N$, we have $\lvert \gamma(P_1) \rvert  \leq \lvert \gamma(u) \rvert $.
    \end{enumerate}
    
    \end{lem}
\begin{proof}

Since  the path $ \intervalcc{v(0,\infty, \pi P_2)}{v(0,\infty, P_2 )} $ is contained in $ M(s_2)$, 
every edge $e \in \edge (\T)$ such that $v(0,\infty, P_2)$ is an extremity of $e$ is an edge of $M(s_2)$.
 For any $Q \in K^\times $ with $\lvert Q \rvert =\lvert P_2 \rvert $ (i.e., $v(0,\infty, Q) = v(0,\infty, P_2)$),
  we have 
 $$\edge(   \intervalco{v(0,\infty, Q)}{Q}  \cap M(s_2)) \neq \emptyset .$$
 Hence,  for  $1 \leq n \leq p-1$,  we have $v(0,\infty, s_2^n(Q))=v(0,\infty, Q)$, i.e.,  $ \lvert s_2^n(Q) \rvert =  \lvert  P_2\rvert $.
 In particular, we have 
 $$ \lvert s_2^n(u) \rvert =  \lvert  P_2\rvert  \text{ and }  \lvert s_2^n(u)- P_2 \rvert =  \lvert  P_2\rvert  .$$
 The equality (2) is satisfied.

  For any $ Q \in K \setminus \lbrace P_1, \dots , P_r \rbrace $, 
     the intersection
     $$   M \cap   \bigcap_{w \in M} \intervalco{w}{Q} $$
    consists of one vertex only, and we denote it  by $\xi(Q)$.
    Since  the half-line 
    $ \intervalco{v(0, \infty , P_2)}{0} $ is contained in $  M,$
    if   
     $$ v(0,\infty ;\xi( Q))  \in \vert(  \intervalco{v(0, \infty , \pi  P_2)}{0}) ,$$ 
   we have $$ v(0,\infty ;\xi( Q)) = v(0, \infty , Q ) \in \vert(  \intervalco{v(0, \infty , \pi P_2)}{0}) ,$$
   hence $\lvert Q \rvert \leq \lvert \pi P_2 \rvert < \lvert P_2 \rvert$.
   In particular, by Lemma \ref{xiij},  for $ Q \in K \setminus \lbrace P_1, \dots , P_r \rbrace $ with $\xi( Q)=\xi_i(j)$ for some distinct elements $i ,j$, we have $\lvert Q \rvert  < \lvert P_2 \rvert$.

       For each $i$,  we put 
   $$A_i : = \lbrace \, Q \in K \setminus \lbrace P_1, \dots , P_r \rbrace  \mid \xi (Q) \in \vert (M(s_i)) \, \rbrace \cup \lbrace P_i \rbrace.$$
  We have $s_2^n (u)  \in A_2 $ since  $ \lvert s_2^n (u) \rvert  =\lvert P_2 \rvert \ (0\leq n \leq p-1) $, $ \lvert P_i \rvert <\lvert P_2 \rvert $ for $i \neq 2$, and $ v(0,\infty ,  P_2 )  \in \vert ( M(s_2))$.

     For $i \neq j $, $ Q \in A_j $, and $  1\leq n \leq p-1 $ , we have $e_i(j) \in \edge(\intervalco{\xi_i(j)}{Q})) $ and
       $  s_i^n(e_i(j)) \not\in  \edge( M)$ by Lemma \ref{graph}.   
    Hence  $M \cap  \intervalco{\xi_i(j)}{s_i^n(Q)})$ consists of $\xi_i(j)$ only. 
    Hence we have $ \xi( s_i^n(Q) ) = \xi_i( j ) \in \vert (M(i)) $.
   In particular,  we have $ s_i^n(Q) \in A_i$.
   
   Since $u \in A_2$ and $\xi_1(2) =v_1$, we have $ \lvert s_1^n (u)  \rvert =1 \ (1 \leq n \leq p-1)$. The equality (3) is satisfied.
   
   Since $P_1 \in A_1$ and $\xi_2(1) =v(0,\infty,\eta)$, we have $ \lvert s_2^n (P_1)  \rvert =\lvert \eta \rvert \ (1 \leq n \leq p-1)$. 
   The equality (1) is satisfied.      
   
      For  an element
  $$ \gamma =s_{i_1}^{n_1} \cdot \cdot \cdot  s_{i_m}^{n_m} \in N  \ (  m \geq 2   , \  1 \leq n_l \leq p-1 \  (1\leq l \leq m) , \  i_l \neq i_{l+1} \  (1\leq l \leq m-1) ),$$
  by the above computations, we have  $ \xi(\gamma (P_i))=  \xi_{i_1}(i_2)$ and $  \xi(\gamma(u))= \xi_{i_1}(i_2)$.
Hence we have $ \lvert  \gamma (P_i) \rvert  <\lvert P_2 \rvert $ for $i \neq 2$,  $ \lvert  \gamma (P_1) \rvert  =\lvert \gamma (u) \rvert $, and $ \lvert  \gamma (u) \rvert <\lvert P_2 \rvert $. In particular, we have $ \lvert  \gamma (u) - P_2 \rvert = \lvert P_2 \rvert $.
Hence (4), (5), and (6)  are satisfied for this $\gamma$.

   For $i \neq 2$, $j \neq i$, and $1 \leq n \leq p-1$, we have $\xi(s_j^n (P_i)) = \xi_{j}(i)$.
    Hence we have $\lvert  s_j^n (P_i) \rvert < \lvert P_2 \rvert$.
    For $i \neq 2$, we also have $\lvert  s_i^n (P_i) \rvert =\lvert P_i \rvert < \lvert P_2 \rvert $ for $0 \leq n \leq p-1$.
    Consequently, the inequality (4) is satisfied for any $\gamma \in N$.
    
    For $j \neq 2$ and $1 \leq n \leq p-1$, we have $\xi(s_j^n (u)) = \xi_{j}(2)$.
    Hence we have $\lvert  s_j^n (u) \rvert < \lvert P_2 \rvert$.
    We also showed that $\lvert s_2^n(u)- P_2 \rvert =  \lvert  P_2\rvert $ for $0\leq n\leq p-1$.
    Consequently, the equality (5) is satisfied for any $\gamma \in N$.

    For $ i \neq 1,2 $,  since  $ \lvert P_1 \rvert  < \lvert  P_2 \rvert $,
        we have 
    $$v(0,\infty; \xi_{i}(1) )  \in \vert ( \intervalco{v(0, \infty ; \xi_{i}(2) )}{0}).$$
   Hence we have $ \lvert s_i^n (P_1) \rvert  \leq \lvert  s_i^n (u) \rvert $.
   Since $s_1(P_1)=P_1=0  $, we have $ \lvert s_1^n(P_1) \rvert <  \lvert s_1^n(u) \rvert \ (0\leq n \leq p-1)$.
 By  (1) and  (2),  we have $ \lvert s_2^n(P_1) \rvert = \lvert \eta \rvert <  \lvert P_2 \rvert  =\lvert s_2^n(u) \rvert \ (1\leq n \leq p-1)$.
      Consequently, the inequality (6) is satisfied for any $\gamma \in N$.     
         \end{proof}
  %

     \begin{center}
    
   \includegraphics[width=12cm]{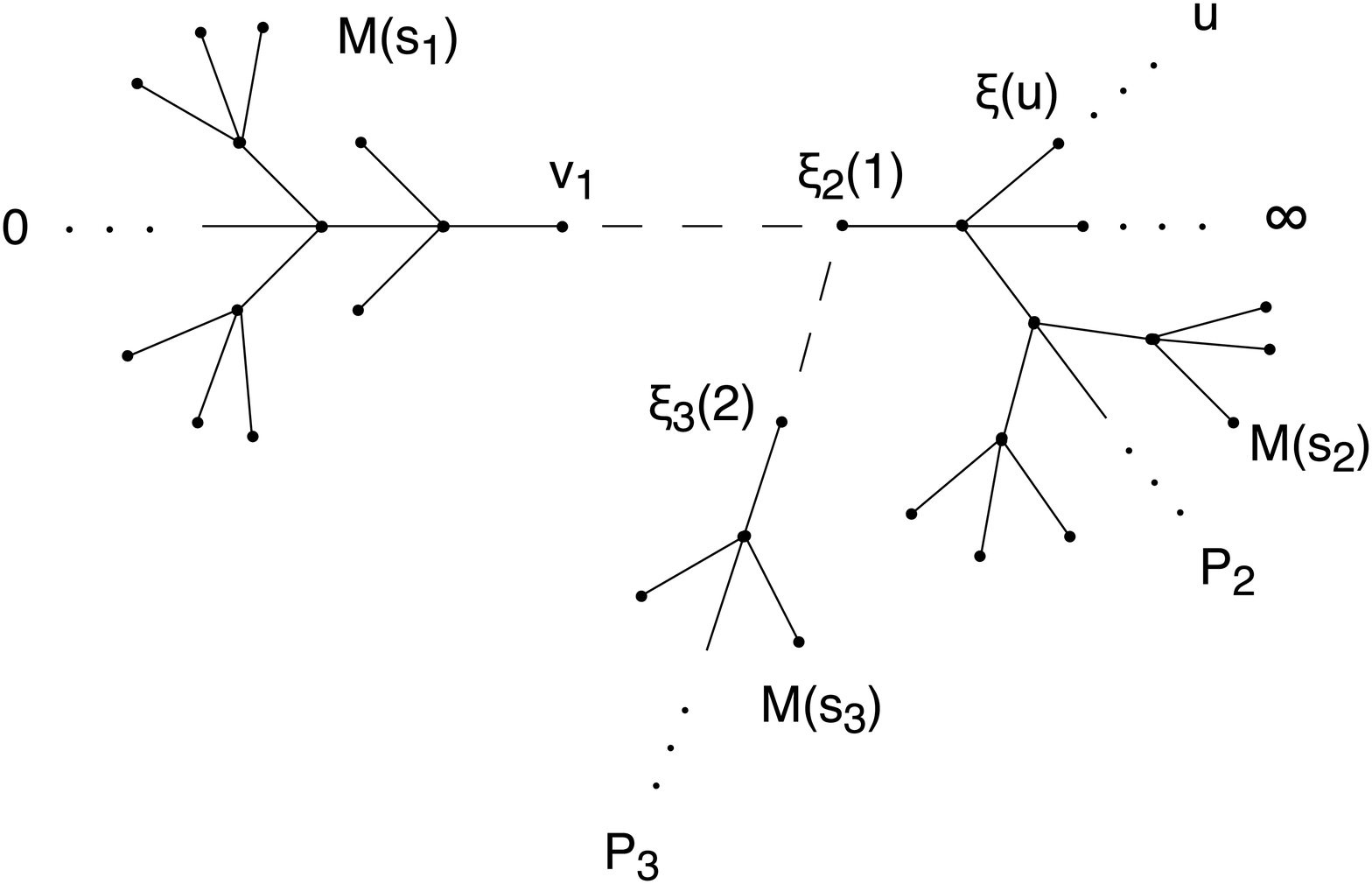}
   
    \textsc{Figure}.\ The subtree $M \subset \T$ generated by $M(s_i) \ (1 \leq i \leq r) $
    \small
    \begin{itemize}
   \item Edges of $M(s_i)$ are denoted  by solid line segments.
   \item Edges of $M \setminus \bigcup_i M(s_i)$ are denoted by dashed line segments.
   \item Half-lines are denoted by  dots.
    \end{itemize} 
     \end{center}
     
           Recall that the function field of $\P^1$ (resp.\ $X$) is denoted by $K(x)$ (resp.\ $F=K(x,y)$).
 We treat $x$ as not only  a function on $X$ and $\P^1$  but also  an $N$-invariant  function on $\Omega$ 
 via the natural projection  $\Omega \to \Omega / N \cong  \P^1$. 
 Similarly, we treat $y$ as not only  a function on $X$  but also  a $\Gamma$-invariant  function on $\Omega$  via the natural projection  $\Omega \to \Omega / \Gamma \cong  \P^1$. 
  
  We also recall that for $1\leq i \leq r$, the image of the fixed point $P_i \in \Omega$ of $s_i$ under the natural projection   $\Omega \to \Omega / N \cong  \P^1$ is the branch point  $a_i \in \P^1$.

For any $\gamma \in N $ and  $ i \neq 2 $, by Lemma \ref{huti} (4), we have $  \lvert \gamma (P_i)\rvert  < \lvert P_2\rvert =  \lvert u \rvert $.   Hence we have  $ \gamma (P_i) \neq u $. We have $x(u) \neq a_i$ for $i \neq 2$.
 By Lemma \ref{huti} (5),  we have $ \gamma (P_2) \neq u$ for any $\gamma \in N $. Hence we have $x (u) \neq a_2$.

There exists $\gamma \in \PGL_2(K)$ such that $\gamma( a_1)=0,\ \gamma( a_2)=1, $ and $ \gamma( x(u))= \infty$. The inverse $\gamma^{-1} $ is written as 
 $$  \gamma^{-1} =  \begin{pmatrix} b & c \\ d & e \end{pmatrix}  \in  \PGL_2(K)$$
for some $ b,c,d,e \in K$
 satisfying $b-a_id \neq 0 \ (1\leq i \leq r).$
 For each $i$, we have 
 \begin{align*}
 \frac{\lambda_i} {x - a_i}
&= \frac{\lambda_i} {\gamma^{-1}(\gamma(x)) - a_i}  \\
 &= \frac{\lambda_i d} {b - a_i d} + \frac{\lambda_i (be-cd)(b-a_id)^{-2}} { \gamma(x) +(c-a_ie)(b-a_id)^{-1}}.
 \end{align*}
By replacing $K$ by its finite extension, there exists $C \in K$ satisfying
$$C^p-C = \sum_{i=1}^r  \frac{\lambda_i d} {b - a_i d} .$$
We have $$(y-C)^p -(y-C) = \sum_{i=1}^r \frac{\lambda_i (be-cd)(b-a_id)^{-2}} { \gamma(x) +(c-a_ie)(b-a_id)^{-1}}.$$
 We also have 
 \begin{align*}
 \frac{c-a_1 e} {b - a_1 d} - \frac{c-a_2 e} {b - a_2 d}
=  \frac{(a_2-a_1 )(be-cd)} {(b - a_1 d)(b-a_2 d)}.
 \end{align*}
Therefore,  the inequality $\lvert \lambda_1 \lambda_2 \rvert < \lvert a_1-a_2 \rvert^2$ is satisfied  if and only if 
 \begin{align*}
\bigg\lvert \frac{\lambda_1 (be-cd)} {(b-a_1d)^2} \frac{\lambda_2 (be-cd)} {(b-a_2d)^2}  \bigg\rvert
 & < \bigg\lvert \frac{(a_2-a_1 )(be-cd)} {(b - a_1 d)(b-a_2 d)} \bigg\rvert^2\\
& = \bigg\lvert \frac{c-a_1 e} {b - a_1 d} - \frac{c-a_2 e} {b - a_2 d} \bigg\rvert^2
  \end{align*}
  is satisfied.
In the rest of this section, by replacing $x$ (resp.\ $y$) by $\gamma(x)$ (resp.\ $y-C$), we may assume 
$ a_1=0,\ a_2 =1$, and $  x(u)= \infty$.
 
 We put 
 $$\alpha :=\prod_{\gamma \in N}  \frac{P_2-\gamma(u)}{P_2-\gamma(P_1)} ,$$
 which converges to an element of $K$; see \cite[Section 8.1]{GerritzenvanderPut:SchottkyMumford}.
  We have $ \lvert  \alpha \rvert =  1$ by Lemma \ref{huti} (4), (5).
  Let $z$ be a coordinate function on $\Omega \subset \P^1$.
   We have 
    $$x(z) = \alpha \prod_{\gamma \in N} \frac{z-\gamma(P_1)}{z-\gamma(u)} $$
  since the both hand sides are $N$-invariant functions on $\Omega$ (i.e., functions on $\P^1 \cong \Omega / N$) having same zeros and poles and being $1 $ at $z=P_2$; see \cite[Section 8.1]{GerritzenvanderPut:SchottkyMumford}.

  We put $$V_{i, \varepsilon} := \lbrace \, z \in \Omega \mid \lvert z-P_i \rvert \leq \varepsilon \, \rbrace $$
  for $i=1,2 $ and $\varepsilon  \in \lvert K^\times \rvert $.
  Since $P_i \in \Omega$ is not a limit point of $N$, by  replacing $K$ by its finite extension and taking  $\varepsilon$ sufficiently small, we may assume $\varepsilon < \lvert P_i - \gamma(u) \rvert$ for any $\gamma \in N$. 
  
   We denote  the power series expansion of $x$ on $V_{i,\varepsilon}$ by
   $$x(z)= \alpha \sum_{n=0}^{\infty} c_{i,n}(z-P_i)^n.$$
 Since $x(P_i)=a_i$, we have    $$x(z) -a_i = \alpha \sum_{n=1}^{\infty} c_{i,n}(z-P_i)^n.$$
   \begin{lem}\label{lambda=}
   We have $ \lambda_1 =   \alpha  c_{1,p}  $ and $ \lambda_2 = (-P_2^2 \eta^{-1})^p   \alpha  c_{2,p}  $.
   \end{lem}
   \begin{proof}
 We put 
 \begin{align*}
 y_1(z) &:= \frac{1}{z} , \\
  y_2(z) &:= - \frac{P_2^2 \eta^{-1}}{z-P_2}.
 \end{align*}
 Then we have $y_i(s_i(z))=y_i(z) +1$ for $i=1,2$. We put  $ f_i :=y_i - y ,$  which is an $s_i$-invariant function on $V_{i,\varepsilon}$.
Since  $y_i $ and $y$ have poles of order $1$ at $P_i$ and we have $P_i= s_i(P_i)$,
the function $f_i$ is holomorphic at $ P_i $.
  We have 
  \begin{align}\label{sikiA}
  y_i^p-y_i = y^p-y +f_i^p-f_i  =\frac{\lambda_i}{x-a_i} + h_i,
  \end{align}
 where we put $$h_i  := f_i^p-f_i  + \sum_{j\neq i } \frac{\lambda_j}{x-a_j},$$ which is holomorphic at $ P_i $.

  For $i=1$, by multiplying the both hand sides of \eqref{sikiA} by $z^p(x-a_1)$, we have 
  $$(x-a_1) -  z^{p-1}(x-a_1) = \lambda_1 z^p + z^p (x-a_1) h_1.$$
  By comparing the degree $1$ terms and the degree $p$ terms with respect to $z$, we have 
   \begin{align*}
      \alpha  c_{1,1} & = 0 , \\
    \alpha     c_{1,p} - \alpha c_{1,1} & = \lambda_1 .
  \end{align*}
  Hence we have $ \lambda_1 =   \alpha  c_{1,p}  $.

For $i=2$,  by multiplying the both hand sides of \eqref{sikiA} by $(z-P_2)^p(x-a_2)$, we have 
\begin{align*}
  &(-P_2^2 \eta^{-1})^p(x-a_2) - (-P_2^2 \eta^{-1}) (z-P_2)^{p-1}(x-a_2) \\
  &= \lambda_2(z-P_2)^p + (z-P_2)^p (x-a_2) h_2.
  \end{align*}
  By comparing the degree $1$ terms and the degree $p$ terms with respect to $z-P_2$, we have 
  \begin{align*}
   (-P_2^2 \eta^{-1})^p   \alpha  c_{2,1} & = 0 , \\
   (-P_2^2 \eta^{-1})^p   \alpha     c_{2,p} -(-P_2^2 \eta^{-1})  \alpha c_{2,1} & = \lambda_2 .
  \end{align*}
  Since $P_2 \neq 0 $, $\eta \in K^\times$, and $\alpha \neq 0$, we have $c_{2,1}=0$.
  Hence we have  $ \lambda_2 = (-P_2^2 \eta^{-1})^p   \alpha  c_{2,p}  $.   
     \end{proof}
     
     \begin{lem}\label{lambdavert}
     We have $ \lvert \lambda_1 \rvert \leq  \lvert \eta \rvert^{p-1} \cdot \lvert P_2 \rvert^{-p}$ and
     $  \lvert \lambda_2 \rvert  \leq \lvert \eta \rvert^{-p} \cdot \lvert P_2  \rvert^p $.
          \end{lem}
  \begin{proof}
  For each $\gamma \in N$,  $n \geq 1$, and $i=1,2$, we put
 \begin{align*}
  & u_{i,0}^{(\gamma)}:= \frac{P_i-\gamma(P_1)}{P_i-\gamma(u)}, \\
  & u_{i,n}^{(\gamma)}:= \frac{1+ u_{i,0}^{(\gamma)}}{ (P_i - \gamma(u) )^n}.
  \end{align*} 
  Since  $\varepsilon < \lvert P_i - \gamma(u) \rvert$ for any $\gamma \in N$,
   we have 
 $$ \frac{z-\gamma(P_1)}{z-\gamma(u)}= \sum_{n=0}^{\infty} u_{i,n}^{(\gamma)} (z-P_i)^n $$
 on $V_{i,\varepsilon}$. (For this calculation, see \cite[Section 3]{Ste83}.)
  Hence we have
  $$x(z)= \alpha \sum_{n=0}^{\infty} c_{i,n}(z-P_i)^n=\alpha \prod_{\gamma \in N} \frac{z-\gamma(P_1)}{z-\gamma(u)}
  = \alpha \prod_{\gamma \in N} \sum_{n=0}^{\infty} u_{i,n}^{(\gamma)} (z-P_i)^n .$$
  
   We shall estimate $\lvert c_{i,p} \rvert $ by calculating  $ \lvert u_{i,n}^{(\gamma)} \rvert $.

    For $i=1$, since $s_1^j(P_1)=P_1$, we have $ u_{1,0}^{(s_1^j)} =0 $ for $ 0 \leq j \leq p-1$. Hence we have  
   $$c_{1,p} = \left( \prod_{\gamma \in N \setminus \{s_1^j \}_{0 \leq j \leq p -1} } u_{1,0}^{(\gamma)} \right) 
  \left( \prod_{j=0}^{p-1} u_{1,1}^{(s_1^j)} \right). $$
  Recall that $P_1=0$.  By Lemma \ref{huti} (6),  we have $ \lvert u_{1,0}^{(\gamma)} \rvert \leq 1$ for  any $\gamma \in N $.
   By   Lemma \ref{huti} (1), (2), we have 
   $\lvert u_{1,0}^{(s_2^j)}  \rvert  =  \lvert \eta  \rvert  \cdot \lvert P_2  \rvert^{-1} $ for $1 \leq j \leq p-1.$
   Since  $u_{1,0}^{(s_1^j)}=0 $ for $1 \leq j \leq p-1$, by  Lemma \ref{huti} (3), we have
  $ \lvert u_{1,1}^{(s_1^j)} \rvert  = \lvert s_1^j(u) \rvert^{-1} =1 $.
  We denote the identity element  of $\PGL_2(K) $ by $\id$.
  Since  $ u_{1,0}^{(\id )}=0 $, 
  we have  $ \lvert u_{1,1}^{(\id)} \rvert = \lvert P_2  \rvert^{-1}$.
    Consequently, we have 
 $$ \lvert c_{1,p} \rvert  \leq \lvert \eta \rvert^{p-1} \cdot \lvert P_2 \rvert^{-p}.$$
  By Lemma \ref{lambda=},  since $\lvert \alpha \rvert =1$,  we have 
    $$  \lvert \lambda_1 \rvert  =\lvert \alpha \rvert \cdot \lvert c_{1,p} \rvert   \leq \lvert \eta \rvert^{p-1} \cdot \lvert P_2  \rvert^{-p} .$$ 
   
For $i=2$, 
by   Lemma \ref{huti} (4), (5),
we have 
$ \lvert u_{2,0}^{(\gamma)} \rvert = 1$ for any $\gamma \in N$.
By this equality and   Lemma \ref{huti} (5), 
we have
\begin{align*}
  \lvert  u_{2,n}^{(\gamma)} \rvert =  \frac{\lvert 1+ u_{2,0}^{(\gamma)} \rvert}{ \lvert P_2 - \gamma(u) \rvert^n }  \leq \lvert P_2 \rvert^{-n}
\end{align*}
for any $\gamma \in N$ and $n \geq 1$.
Therefore,  we have $$ \lvert c_{2,p} \rvert \leq \lvert P_2 \rvert^{-p}.$$
By Lemma \ref{lambda=},  since $\lvert \alpha \rvert =1$,  
we have
    $$  \lvert \lambda_2 \rvert  =\lvert P_2\rvert^{2p} \cdot \lvert \eta \rvert^{-p} \cdot \lvert \alpha \rvert \cdot \lvert c_{2,p} \rvert  
      \leq \lvert \eta \rvert^{-p} \cdot \lvert P_2  \rvert^{p} .$$ 
  \end{proof}

 
           
     By Lemma \ref{v} and Lemma \ref{lambdavert}, we have
\begin{align*}
  \lvert \lambda_1 \lambda_2 \rvert 
   \leq \lvert \eta \rvert^{-1} 
   < 1 
   = \lvert a_1-a_2 \rvert^2.
\end{align*}
   Recall that we have assumed $a_1=0$ and $a_2=1$.
   
   Theorem \ref{main} follows from this result and  the result  of Section \ref{if}.

   \subsection*{Acknowledgements}
   The author would like to express his deepest gratitude to his adviser  Tetsushi Ito for considerable and  invaluable guidances.
   He is also grateful to  Takahiro Tsushima for helpful comments about computation of reductions.


\begin{thebibliography}{99}

\bibitem{BGR}
\textsc{S. Bosch, U. G\"untzer and R. Remmert}, Non-Archimedean Analysis: A Systematic Approach to Rigid Analytic Geometry, Grundlehren der mathematischen Wissenschaften \textbf{261}, Springer, Berlin, 1984.

\bibitem{Bra07} 
\textsc{P.\ E. Bradley}, Cyclic coverings of the $p$-adic projective line by Mumford curves,
  Manuscripta Math.\ \textbf{124} (2007), no.\ 1, 77-95.
  
  \bibitem{CorKatoKont01} 
\textsc{G. Cornelissen, F. Kato and A. Kontogeorgis},  Discontinuous groups in positive characteristic and automorphisms of Mumford curves, Math.\ Ann.\ \textbf{320} (2001), no.\ 1, 55-85.
  
\bibitem{Fresnel-vanderPut:2004}
 \textsc{J. Fresnel and M.  van der Put},  Rigid analytic geometry and its applications,
  Progress in Mathematics \textbf{218},\ Birkh\"auser, Boston, 2004.

\bibitem{GerritzenvanderPut:SchottkyMumford} 
\textsc{L. Gerritzen and M.  van der Put},  Schottky groups and Mumford curves, Lecture Notes in Mathematics \textbf{817},\ Springer, Berlin, 1980. 

\bibitem{Lutk} 
 \textsc{W. L\"utkebohmert}, Rigid geometry of curves and their Jacobians,
  Ergebnisse der Mathematik und ihrer Grenzgebiete
  \textbf{61},\ Springer, Cham, 2016.

\bibitem{Mum} 
 \textsc{D. Mumford},  An analytic construction of degenerating curves over complete local rings,
  Compositio Math.\ \textbf{24} (1972), 129-174.

\bibitem{Tree}  
  \textsc{J.-P. Serre},  Trees,  Corrected 2nd printing of the 1980 English translation,\ Springer Monographs in Mathematics,\ Springer, Berlin, 2003.

\bibitem{Ste82} 
  \textsc{G. van Steen},  Galois coverings of the non-Archimedean projective line,
  Math.\ Z.\ \textbf{180} (1982), no.\ 2, 217-224.

\bibitem{Ste83} 
\textsc{G. van Steen}, Non-Archimedean Schottky groups and hyperelliptic curves,
 Nederl.\ Akad.\ Wetensch.\ Indag.\ Math.\ \textbf{45} (1983), no.\ 1, 97-109.

\end{thebibliography}
\end{document}